\documentclass{birkjour}
\usepackage{amsmath}
\usepackage{amssymb}
\usepackage{color}
\usepackage{mathrsfs}
\usepackage{cases}
\usepackage{amsthm}
\newtheorem{theorem}{Theorem}[section]
\newtheorem{corollary}[theorem]{Corollary}
\newtheorem{lemma}[theorem]{Lemma}

\newtheorem{definition}[theorem]{Definition}
\newtheorem{remark}[theorem]{Remark}

\theoremstyle{definition}

\numberwithin{equation}{section}

\begin{document}
\title[Spectral characterization of Q-P-D functions on the real line]{Spectral characterization of quaternionic positive definite functions on the real line}

\author[Z. P. Zhu]{Zeping Zhu}
\email{zzp$\symbol{64}$mail.ustc.edu.cn}

\address[Z. P. Zhu]{Department of Mathematics, Chongqing Normal University, Chongqing 401331, China}

\begin{abstract}This paper is concerned with the spectral characteristics of quaternionic positive definite functions on the real line. We generalize the Stone's theorem to the case of a right quaternionic linear one-parameter unitary group via two different types of functional calculus. From the generalized Stone's theorems we obtain a correspondence   between  continuous quaternionic positive definite functions and spectral systems, i.e., unions of a spectral measure and a unitary anti-self-adjoint operator that commute with each other; and then deduce that the Fourier transform of a continuous quaternionic positive definite function is an unusual type of quaternion-valued measure which can be described equivalently in two different ways. One is related to spectral systems on $\mathbb{R}^+$(induced by the first generalized Stone's theorem), the other is related to non-negative finite Borel measures on $\mathbb{R}^3$(induced by the second generalized Stone's theorem).   An application to weakly stationary quaternionic random processes is also presented. \end{abstract}

\thanks{This work was supported by National Science Foundation for Youth (NSFY) of China (No. 12101094).}



\maketitle

\section{Introduction}

The notion of positive definiteness is an important and basic notion in mathematics, which occurs in a variety of algebraic settings. At present there exists a rather satisfactory theory of complex-valued positive definite functions on abelian semigroups. The further development  splits mainly into two orientations: the theory of positive definite functions on non-abelian semigroups, and that of positive definite functions taking values in non-commutative algebras. This paper is concerned with the latter. Our interest lies especially in the spectral properties of quaternionic positive definite functions. The first contribution on this topic has been made by D. Alpay, F. Colombo, D. P. Kimsey, and I. Sabadini (see, e.g., \cite{Alpay-2016-1}). In addition, there are some other related contributions on quaternionic positive definiteness (see, e.g., \cite{Alpay-2016-3,Alpay-2017,Alpay-2004}).

There are two major difficulties:

Firstly, the conceptual framework in the quaternionic case is incomplete. Many vital concepts in the classical theory of positive definite functions have no proper counterparts in the quaternionic setting. For instance, in the complex case, for every locally compact group $G$, its Pontryagin dual, which is composed of continuous group homomorphisms from $G$ to the unit circle in the complex plane, has a natural group structure given by the ordinary function multiplication, thus it is also called the dual group of $G$; whereas, the analog composed of continuous group homomorphisms from $G$ to the unit sphere in the real quaternion algebra, possesses no natural group structure due to the non-commutative nature of quaternions. To overcome this  difficulty, D. Alpay and his collaborators introduced a special type of quaternion-valued measure, called q-positive measure, on the classical Pontryagin dual with respect to an arbitrary complex slice of the quaternion algebra (see, e.g., \cite{Alpay-2016-1}). Then they established a one-to-one correspondence between quaternionic positive definite functions and this special kind of measure via the Fourier-Stieltjes transform. We would like to mention that our approach is different, and our results are strongly connected with theirs.

Secondly, as is well known, several branches of the functional analysis play crucial roles in the theory of positive definite functions. In contrast to the complex case, the functional analysis in quaternionic vector spaces still remains to be perfected; it causes the absence of some important analytical tools. Recent contributions (see, e.g., \cite{Alpay-2016,Alpay-2016-2,Colombo-2007,Ghiloni-2017,Ghiloni-2013,Ludkovsky-2013}) on the normal operators in quaternionic Hilbert spaces provide us with two types of functional calculus. One is based on the right spectrum, also called S-spectrum, the other is based on the left spectrum. We attempt to apply both the functional calculuses to discuss the spectral characteristics of quaternionic positive definite functions on $\mathbb R$, one of the simplest locally compact abelian groups.

Our strategy is as follows:

First we establish two generalized Stone theorems for  right quaternionic linear one-parameter unitary groups via two different types of functional calculus in quaternionic Hilbert spaces. Explicitly speaking, the generalized Stone theorems state that
every strongly continuous right linear one-parameter unitary group $U(t)$ can be expressed as
$$U(t)=e^{tA}\rvert_S,$$
and
$$U(t)=e^{tA}\rvert_L,$$
for all $t\in\mathbb R$ with $A$ being a normal operator. Here $e^{tA}\rvert_S$ is defined as the S-functional calculus for the function $e^{tx}$ on the S-spectrum of $A$, and $e^{tA}\rvert_L$ is defined as the other type of functional calculus on the left spectrum.

Based on the first generalized Stone theorem we construct a correspondence between  continuous quaternionic positive definite functions and spectral systems, namely, unions of a spectral measure and a unitary anti-self-adjoint operator that commute with each other (for more details, one may refer to Theorem \ref{Thm-Spectral}). It leads to the conclusion that the Fourier transform of a continuous quaternionic positive definite function is a slice-condensed measure, which is an unusual type of quaternionic valued measure related to spectral systems (see Definition \ref{Def-PNM-2}). More precisely, if $\varphi$ is a continuous quaternionic positive definite function on $\mathbb R$, then there exists a unique slice-condensed measure $\mu$ such that
$$\varphi(t)=\int_{\mathbb{R}^+}\cos(tx)d\mathbf{Re}\mu(x)+\int_{\mathbb{R}^+}\sin(tx)d(\mu-\mathbf{Re}\mu)(x), $$
and vice versa. After that we apply the second generalized Stone theorem to show the concept of slice-condensed measure can be defined equivalently (as Definition \ref{Def-PNM-1}) in a more concrete way: A quaternionic regular Borel measure $\mu$ on $\mathbb{R}^+$ is  slice-condensed if and only if there exists a non-negative finite regular Borel measure $\Gamma$ on $\mathbb{H}_I$, namely the 3-dimensional real vector space of pure imaginary quaternions, s.t.,
  $$\mu=\rho_*(\Gamma+\frac{x}{\lvert x\rvert }\Gamma),$$
where $\rho_*$ is the push-forward mapping induced by the function $\rho:x\mapsto \lvert x\rvert , \ x\in\mathbb{H}_I$.

The present paper is organized as follows: Some preliminaries are given in Section \ref{Sec-Pre}. Two generalized  Stone's theorems for right quaternionic linear unitary groups are established in Section \ref{Sec-Stone}. We devote Section \ref{Sec-Bochner} to the spectral characteristics of quaternionic positive definite functions on the real line; especially a generalized Bochner's theorem is established in this section. An application to weakly stationary quaternionic random processes is presented in Section \ref{Sec-App}. Section \ref{Sec-Final} is the final remark.

\section{Preliminaries}\label{Sec-Pre}
We would like to introduce some basic information about two types of functional calculus in quaternionic Hilbert spaces.

Let $\mathbb{H}$ denote the real quaternion algebra
$$\{q=q_0i_0+q_1i_1+q_2i_2+q_3i_3,\ q_i\in\mathbb{R}\ (i=0,1,2,3)\},
$$
where $i_0=1$, $i_3=i_1i_2$, and $i_1,i_2$ are the generators of  $\mathbb H$, subject to the following identities:
 $$i_1^2=i_2^2=-1, \qquad i_1i_2=-i_2i_1. $$
For all $q\in \mathbb{H}$,
its conjugate is defined as $\overline{q}:=q_0i_0-q_1i_1-q_2i_2-q_3i_3$, and its norm given by $|q|:=\sqrt{q_0^2+q_1^2+q_2^2+q_3^2}$.
$\mathbb{S}$ will denote the set of all imaginary units, namely,
$$\mathbb{S}:=\{q=q_0i_0+q_1i_1+q_2i_2+q_3i_3\in\mathbb{H}:q_0=0\text{ and }|q|=1\}. $$
Consider the subalgebra $\mathbb{C}_j$ generated by a imaginary unit $j\in\mathbb{S}$. It can be easily seen that $\mathbb{C}_j$ in fact is a complex field since $j^2=-1$. Let $\mathbb{C}_j^+$ denote  the set of all $p\in\mathbb{C}_j$ with $\mathbf{Im}p\geq 0$, i.e., $$\mathbb{C}_j^+:=\{q=q_0i_0+q_jj\in \mathbb{C}_j: q_0\in\mathbb R,  q_j\geq 0 \}. $$

Let $V$ be a right vector space over $\mathbb{H}$. An inner product on $V$ is a map $\langle\cdot, \cdot\rangle: V\times V \mapsto \mathbb{H}$ with the following properties:
\begin{eqnarray*}
  \langle x, y\rangle&=& \overline{\langle y, x\rangle}, \\
  \langle x+y,z\rangle&=&\langle x,z\rangle+\langle y,z\rangle, \\
  \langle xp, y\rangle&=&\langle x, y\rangle p, \\
  \langle x, yp\rangle&=&\overline{p}\langle x, y\rangle, \\
\end{eqnarray*}
and
$$\langle x, x\rangle \geq 0, \ =0 \text{ if and only if }x=0,
$$
for all $x, y, z\in V$ and $p\in \mathbb{H}$.
If $\langle\cdot, \cdot\rangle$ is an inner product, then $\|x\|=\sqrt{\langle x, x\rangle}$ is norm on $V$. A right vector space $V$ over $\mathbb{H}$ endowed with an inner product which makes $V$ be a complete norm space  is called a quaternionic Hilbert space (see, e.g., \cite{Alpay-2016,Ghiloni-2013}).

Let $\mathcal{H}$ be a quaternionic Hilbert space. The set of all right linear bounded operators on $\mathcal{H}$ will be denoted by $\mathcal{B}(\mathcal{H})$, and the set of all right linear operators on the subspaces of  $\mathcal{H}$ by $\mathcal{L}(\mathcal{H})$. For any operator $T$, the definition domain, the range and the kernel will be denoted by $\mathrm{D}(T)$, $\mathrm{R}(T)$ and $\mathrm{Ker}(T)$ respectively. The concepts of unitary, normal, self-adjoint and anti-self-adjoint operators are defined in the same way as the case that  $\mathcal{H}$ is a real or complex Hilbert space.

\subsection{Functional calculus based on the S-spectrum}

For a densely defined operator $T\in \mathcal{L}(\mathcal{H})$, its S-spectrum (see Definition 2.12. of \cite{Alpay-2016}) is defined as
$$\sigma_S(T):=\mathbb{H}\setminus \rho_S(T), $$
where $\rho_S(T)$ is the S-resolution set of $T$ given by
 \begin{eqnarray*}
    \rho_S(T)&:=& \big\{q\in\mathbb{H}:\mathrm{Ker}(\mathcal{R}_q(T))=0,\ \mathrm{R}(\mathcal{R}_q(T)) \text{ is
                    dense in } \mathcal{H}  \\
             &  &\text{ and } \mathcal{R}_q(T)^{-1}\in \mathcal{B}(\mathcal{H})\big\}
  \end{eqnarray*}
with $\mathcal{R}_q(T):=T^2-2\mathbf{Re}(q)T+|q|^2I$.

A resolution of the identity in a quaternionic Hilbert space is defined as follows.
\begin{definition}\label{Def-right-linear-ROI}
  Let $\mathcal{M}$ be the $\sigma$-algebra of all Borel sets on a locally compact Hausdorff space  $\Omega$, and $\mathcal H$ be a quaternionic Hilbert space. A  resolution of the identity on $\Omega$  is a mapping
  $$E: \mathcal{M}\mapsto\mathcal{B}(\mathcal{H})$$
  with the following properties:

  1. $E(\emptyset)=0$, $E(\Omega)=I$.

  2. $E(\omega)$ is a right linear self-adjoint projection for all $\omega\in\mathcal{M}$.

  3. $E(\omega'\cap \omega'')=E(\omega')E(\omega'')$ holds for all $\omega',\ \omega''\in\mathcal{M}$.

  4. If $\omega'\cap \omega''=\emptyset$, then $E(\omega'\cup \omega'')=E(\omega')+E(\omega'')$.

  5. For every $x, y\in\mathcal{H}$, the set function $E_{x,y}$ defined by
  $$E_{x,y}(\omega)=\langle E(\omega)(x),y\rangle$$
  is a quaternion-valued regular Borel measure on $\Omega$.
\end{definition}

\begin{theorem}\label{Thm-Spetral-Normal}\cite{Alpay-2016}
Let $T$ be a right linear normal operator on a quaternionic Hilbert space $\mathcal{H}$ and $j$ be an imaginary unit in $\mathbb{S}$. There exists a uniquely determined resolution of the identity $E_j$ on $\sigma_S(T)\cap \mathbb{C}_j^+$, such that
$$\langle Tx,y\rangle
    =\int_{\sigma_S(A)\cap \mathbb{C}_j^+}
      \mathbf{Re}(p)d \langle E_jx,y\rangle (p)+
     \int_{\sigma_S(A)\cap \mathbb{C}_j^+}
      \mathbf{Im}(p)d\langle JE_jx,y\rangle(p),
$$
for all $x\in \mathrm{D}(T)$ and all $y\in\mathcal{H}$,
where $J\in\mathcal{B}(\mathcal{H})$ is a unitary  anti-self-adjoint operator associated with $T$.
\end{theorem}
$E_j$ is called the spectral measure of $T$. Note that $E_j$ and $J$ commute with each other.

\begin{definition}\label{Def-FCS}
  Let $\mathcal{H}$ be a quaternionic Hilbert space. If  a resolution of the identity $E$ on a locally compact Hausdorff space $\Omega$ commutes with a  unitary  anti-self-adjoint operator $J\in\mathcal{B}(\mathcal{H})$, then $(E,J)$ is called a spectral system on $\Omega$.
\end{definition}

From this point of view, $(E_j,J)$ in Theorem \ref{Thm-Spetral-Normal} (Theorem 6.2  of \cite{Alpay-2016}) is a spectral system on $\sigma_S(A)\cap \mathbb{C}_j^+$.

\begin{definition}\cite{Alpay-2016}
 A  subset $\Omega$ of $\mathbb{H}$ is said to be axially symmetric if it satisfies the following property:
 For an arbitrary element $p_0+p_1j$ $(p_0, p_1\in \mathbb{R},j\in\mathbb{S})$ in $\Omega$,
 $p_0+p_1j'$ also belongs to $\Omega$ for all $j'\in\mathbb{S}$.
\end{definition}

\begin{definition}\cite{Alpay-2016}
  Let $\Omega$ be an  axially symmetric subset of $\mathbb{H}$. Set
  $$D:=\{(u,v)\in\mathbb{R}^2: u+vj\in\Omega \text{ for some }j\in\mathbb{S}\}. $$ A
  function $f:\Omega\mapsto\mathbb{H}$ is called an intrinsic slice function if it can be composed as
  $$f(u+vj)=f_0(u,v)+f_1(u,v)j, \quad \forall \ u,v\in D, \ \forall\ j\in\mathbb{S}
  $$
  where $f_0$ and $f_1$ are both real-valued functions defined on $D$.
\end{definition}

The functional calculus based on the S-spectrum, also called S-functional calculus, is defined as follows.
\begin{definition}\cite{Alpay-2016}\label{Def-FC-S-Spec}
Let $T$ be a right linear normal operator on a quaternionic Hilbert space $\mathcal{H}$, and $(E_j,J)$ be the spectral system that arises in Theorem \ref{Thm-Spetral-Normal}. For any intrinsic slice function $f: \sigma_S(T)\mapsto\mathbb{H}$ with the real component $\mathbf{Re}(f)$ and the imaginary component $\mathbf{Im}(f)$ both bounded and Borel measurable, the S-functional calculus for $f$ is defined by
\begin{equation*}\begin{split}
     \langle f(T)x,y\rangle = & \int_{\sigma_S(A)\cap \mathbb{C}_j^+}
      \mathbf{Re}(f(p))d \langle E_jx,y\rangle(p) \ + \\
       & \int_{\sigma_S(A)\cap \mathbb{C}_j^+}
      \mathbf{Im}(f(p))d\langle JE_jx,y\rangle(p)
\end{split}
\end{equation*}
for all $x, y\in\mathbb{H}$
\end{definition}

\begin{remark}
 The S-spectrum was discovered by F. Colombo and I. Sabadini in 2006. Since then the theory of S-functional calculus developed very rapidly. And now it has been quite completed by a series of important contributions (see, e.g., \cite{Alpay-2016,Alpay-2016-2,Colombo-2007,Colombo-2008,Colombo-2009,Colombo-2011})
  There are two other types of functional calculus in quaternionic Hilbert spaces (see \cite{Ghiloni-2018,Viswanath-1971}) similar with the S-functional calculus. The  functional  calculus via intertwining quaternionic PVMs \cite{Ghiloni-2018} and  the S-functional  calculus of unbounded operators \cite{Alpay-2016}  are both based on the continuous functional calculus introduced by  R. Ghiloni, V. Moretti, and A. Perotti \cite{Ghiloni-2013} as pointed out by their authors. The functional calculus via spectral systems \cite{Viswanath-1971} has been established decades earlier than the other, but no proper notion of quaternionic spectrum appears. So from our point of view, maybe not very accurate, the  S-functional calculus is an elegant completion of the functional calculus via spectral systems.
\end{remark}

\subsection{Functional calculus based on the left spectrum}
In order to give a spectral characterization for quaternionic positive definite functions, we need apply two types of functional calculus, namely, the functional calculus based on the S-spectrum \cite{Alpay-2016} and the functional calculus based on the left spectrum \cite{Ludkovsky-2013,Ludkovsky-2012}, to  certain quaternionic Hilbert spaces.

Let $\mathcal H$ be a quaternionic Hilbert space, i.e., a right $\mathbb{H}$-vector space with an inner product $\langle\cdot,\cdot\rangle$. Then the functional calculus based on the S-spectrum can  be established on $\mathcal H$ as shown in \cite{Alpay-2016}. However, we can't apply the functional calculus based on the left spectrum directly to $\mathcal H$, since the basic setting in \cite{Alpay-2016} is different from that in \cite{Ludkovsky-2013,Ludkovsky-2012}. So we must make some adjustment as follows:

1. Introduce a new inner product $(\cdot\mid \cdot)$, given by
$$(x \mid  y)=\overline{\langle x,y\rangle}, \quad \forall \ x, y\in\mathcal{H}.$$

2. Construct a left $\mathbb{H}$-linear structure on $\mathcal{H}$: Let $\{x_a\}_{a\in\Sigma}$ be an orthonormal basis, the left scalar multiplication is defined as
$$qx=\sum_{a\in\Sigma}x_aq\langle x,x_a\rangle, \quad \forall\ q\in \mathbb{H}, \ x\in \mathcal{H}. $$
Then the $\mathbb{H}$-vector space $\mathcal{H}$ endowed with the inner product $(\cdot\mid\cdot)$  can be treated as a quaternionic Hilbert space defined in \cite{Ludkovsky-2013,Ludkovsky-2012}. We would like to emphasize that this very type of quaternionic Hilbert space in \cite{Ludkovsky-2013,Ludkovsky-2012} will be called a bilateral quaternionic Hilbert space in our paper to distinguish it from the other type of  quaternionic Hilbert space introduced in the preceding subsection.

For convenience, let $\tilde{\mathcal{H}}$ denote the bilateral quaternionic Hilbert space transformed from a quaternionic Hilbert space $\mathcal{H}$. An operator $T$ on $\tilde{\mathcal{H}}$ is said to be quasi-linear if it is additive and $\mathbb{R}$-homogeneous, i.e.,
  $$T(x+y)=T(x)+T(y), \quad \forall \ x, y\in \mathrm{D}(T),$$
  $$T(qx)=T(xq)=qT(x),\quad \forall \ x\in \mathrm{D}(T), \ q\in\mathbb{R}. $$
  The Banach space consisting of all bounded quasi-linear operators is denoted by $\mathcal{B}_q(\tilde{\mathcal{H}})$. Note that every bounded right linear operator on $\mathcal{H}$
 is a bounded quasi-linear operator on $\tilde{\mathcal{H}}$, which is to say
 $$\mathcal{B}(\mathcal{H})\subset\mathcal{B}_q(\tilde{\mathcal{H}}). $$

\begin{definition}\cite{Ludkovsky-2013}
  The left spectrum, denoted by $\sigma_L(T)$,  of a closed densely defined quasi-linear operator $T$ is the set of all $q\in\mathbb{H}$ such that $T-qI$ is not bijective from the definition domain $\mathrm{D}(T)$ onto the whole space $\tilde{\mathcal{H}}$.
\end{definition}

As shown in \cite{Ludkovsky-2013}, for any (not necessarily bounded) normal operator $T$, there exists a unique smallest  quasi-commutative von Neumann algebra $\mathbf{A}\subset \mathcal{B}_q(\tilde{\mathcal{H}})$ that $T$ is affiliated with; moreover any quasi-commutative von Neumann algebra is $*$-isomorphic to $C(\Lambda,\mathbb{H})$ for some compact Hausdorff space $\Lambda$ via a generalized Gelfand transform. Then the $*$-isomorphism from $C(\Lambda,\mathbb{H})$ to $\mathbf{A}$ induces  an involution preserving bounded $\mathbb{H}$-algebra homomorphism $$\phi: \mathcal{B}(\sigma_L(T),\mathbb{H})\mapsto\mathcal{B}_q(\tilde{\mathcal{H}}), $$ where $\mathcal{B}(\sigma_L(T),\mathbb{H})$ denotes the algebra consisting of all $\mathbb{H}$-valued bounded Borel measurable functions defined on $\sigma_L(T)$. Then the functional calculus on the left spectrum is naturally defined as follows.

\begin{definition}\label{Def-FC-L-Spec}\cite{Ludkovsky-2013}
The functional calculus on the left spectrum is given by
$$f(T)=\phi(f), \quad \forall \ f\in\mathcal{B}(\sigma_L(T),\mathbb{H}).$$
Moreover, there exist regular $\mathbb{R}$-valued Borel measures $\mu_{i_v,i_l}[x,y]$ with  $v,l=0,1,2,3$ and $ x,y\in\tilde{\mathcal{H}}$ such that
\begin{equation*}
(f(T)x\mid y)=\sum_{v,l=0}^{3}\int_{\sigma_L(T)}f_vi_l\ d\mu_{i_v,i_l}[x,y],
\end{equation*}
where $f=\sum_{v=0}^{3}f_vi_v$ and $f_v$ is $\mathbb{R}$-valued.
\end{definition}

\begin{remark}
  Another generalized Gelfand transform in the quaternionic case has been investigated by S. H. Kulkarni \cite{Kulkarni-1994,Kulkarni-1992} quite earlier. By contrast, the theory established by S. V. Ludkovsky  has a higher degree of completion $($one may refer to \cite{Ludkovsky-2013,Ludkovsky-2012} for more details$)$.
\end{remark}

To avoid misunderstanding, we would like to mention the following facts:

1. Whether $V$ is a quaternionic Hilbert space or a bilateral quaternionic Hilbert space, the real part of its inner product $(\cdot\mid \cdot)$ is a real inner product. For a densely defined operator $T$ on $V$, its adjoint  $T^*$ in $(V,(\cdot\mid \cdot))$ is identical with its adjoint  in the real Hilbert space $(V,\mathbf{Re}(\cdot\mid \cdot))$.
Furthermore, $$(Tx\mid y)=(x\mid T^*y),\qquad x\in\mathrm{D}(T),y\in\mathrm{D}(T^*)$$ holds when $T$ is right linear. But this equality  may fail when $T$ is quasi-linear.

For example, let's observe the operators $L_q$ and $R_q$, i.e., the left and right scalar multiplication by $p\in\mathbb{H}$, on a bilateral quaternionic Hilbert space $V$. It can be easily verified that $L_q^*=L_{\overline{q}}$, $R_q^*=R_{\overline{q}}$, and $(L_qx\mid y)=(x\mid L_{\overline{q}}y)$ holds for all $x,y\in V$. In contrast,  $(R_qx\mid y)=(x\mid R_{\overline{q}}y)$ is generally not valid. The key difference between $L_q$ and $R_q$ is that the former is right linear, while the later is not.

2. The left scalar multiplication in a (bilateral) quaternionic Hilbert space is often uncertain. It may cause some problems since the left spectrum depends on the left scalar multiplication.

For instance, assume that we are discussing two closed densely defined operator $A$ and $B$; in one situation we may discover that there exists $q\in\mathbb{H}$ so that
$A=qB$, then it follows naturally that $\sigma_L(A)$ is identical with $q\sigma_L(B)$; however, in another situation, if the left scalar multiplication changes, the equality
$$\sigma_L(A)=q\sigma_L(B)$$
 may no longer hold true.

The uncertainty of left scalar multiplication has been mentioned in Section 1 of \cite{Ghiloni-2017}, and also can be observed in Lemma 3.5 and Theorem 3.6 in \cite{Alpay-2016}.

\section{Stone's theorems in quaternionic Hilbert spaces}\label{Sec-Stone}

In this section, we are going to apply both the functional calculuses based on the S-spectrum and the left spectrum to establish two generalized Stone's theorems for one-parameter unitary groups in quaternionic Hilbert spaces.
For precision, if $f(T)$ is given by the functional calculus based on the S-spectrum,  we denote it by $f(T)\rvert_S$; if it is given by the functional calculus based on the left spectrum, then  denote it by $f(T)\rvert_L$.
Only when there is no ambiguity, we will just write it as $f(T)$.

\begin{definition}\label{Def-OPUG}
  A one-parameter unitary group on a quaternionic Hilbert space $\mathcal{H}$ is a family $U(t)$, $t\in\mathbb{R}$, of right linear unitary operators on $\mathcal{H}$ with the following properties:
  $$U(0)=I,\quad U(s+t)=U(s)U(t)\ \text{for all } s,t\in \mathbb{R}.$$
  A one-parameter unitary group is said to be strongly continuous if
  \begin{equation}\label{Eq-Def-OPUG-Continuity}
    \lim_{s\to t}\|U(s)(x)-U(t)(x)\|=0
  \end{equation}
  for all $t\in \mathbb{R}$  and all $x\in\mathcal{H}$.
\end{definition}

\begin{definition}\label{Def-OPUG-GENERATOR}
  Let $U(t)$ be a strongly continuous one parameter unitary group on $\mathcal{H}$. The infinitesimal generator of $U(t)$ is the operator $A$ defined by
  \begin{equation}\label{Eq-Def-OPUG-genarator}
    A(x):=\lim_{t\to 0}\frac{U(t)(x)-x}{t},
  \end{equation}
  with its domain $\mathrm{D}(A)$ consisting of all $x\in \mathcal{H}$ for which the limit exists in the norm topology on $\mathcal{H}$.
\end{definition}
This type of infinitesimal generator has been investigated in the study of semigroups over real algebras (see, e.g., \cite{Colombo-2011,Ghiloni-2018}). They are anti-self-adjoint in contrast to the classical infinitesimal generators.

\begin{theorem}\label{Thm-Stone-I}
  Suppose $U(t)$ is a strongly continuous one-parameter unitary group on $\mathcal{H}$. Then the infinitesimal generator $A$ is a right linear anti-self-adjoint operator, and
  $$U(t)=e^{tA}\rvert_S \ \text{for all} \ t\in \mathbb{R}.$$
\end{theorem}

Here $e^{tA}\rvert_S$ is defined as the functional calculus for the intrinsic slice function $e^{tx}$ on the S-spectrum of $A$. To be precise, Definition \ref{Def-FC-S-Spec} says
  \begin{equation*}
   \begin{split}
       \langle e^{tA}\rvert_Sx,y\rangle
    =& \int_{\sigma_S(A)\cap \mathbb{C}_j^+}
      \mathbf{Re}(e^{tp})d \langle E_jx,y\rangle (p) \ + \\
         &  \int_{\sigma_S(A)\cap \mathbb{C}_j^+}
      \mathbf{Im}(e^{tp})d\langle JE_jx,y\rangle (p),
    \end{split}
  \end{equation*}
where $(E_j,J)$ is the spectral system associated with $A$. Moreover, we have
$$\mathbf{Re}(e^{tp})=\cos(t\lvert p\rvert)
  \ \text{ and }\
  \mathbf{Im}(e^{tp})=\sin(t\lvert p\rvert),
$$
since  $\sigma_S(A)$ is a subset of $\mathbb H_I$, namely, the 3-dimensional real vector space of pure imaginary quaternions.

\begin{theorem}\label{Thm-Stone-II}
  Suppose $U(t)$ is a strongly continuous one-parameter unitary group on $\mathcal{H}$. Then
  $$U(t)=e^{tA}\rvert_L \ \text{for all} \ t\in \mathbb{R},$$
  where $A$ is the infinitesimal generator of $U(t)$.
\end{theorem}

Note that $e^{tA}\rvert_L$ is given as the functional calculus  for the function $e^{tx}$ on the left spectrum of $A$ in the bilateral Hilbert space $\tilde{\mathcal{H}}$ transformed from $\mathcal{H}$.
Moreover, according to Definition \ref{Def-FC-L-Spec}, we have
\begin{equation}\label{Eq-Stone-II-integral}
\begin{split}
   (e^{tA}\rvert_Lx\mid y)= & \sum_{l=0}^{3}\int_{\sigma_L(T)}\cos(\lvert p\rvert )i_l\ d\mu_{i_v,i_l}[x,y](p)\ +\\
                & \sum_{v=1}^{3}\sum_{l=0}^{3}\int_{\sigma_L(T)}\frac{p_v}{\lvert p\rvert }\sin(\lvert p\rvert )i_l\ d\mu_{i_v,i_l}[x,y](p),
\end{split}
\end{equation}
where $p\in\sigma_L(A)$ is composed as $p=p_1i_1+p_2i_2+p_3i_3$ with $p_v\in\mathbb{R}$, and $\mu_{i_v,i_l}[x,y]$ are the regular Borel measures on $\sigma_L(A)$ uniquely determined by $A$ and $x,y\in\tilde{\mathcal{H}}$.

A generalization of Stone's theorem to the case of a one-parameter unitary group in a quaternionic Hilbert space has been established by  S. V. Ludkovsky (see Theorem 2.33 in \cite{Ludkovsky-2007}).  However, this version of Stone's theorem does not fit in with our aims, because the infinitesimal generator defined  in \cite{Ludkovsky-2007} is self-adjoint.

We would like to emphasize that the major difference between  the earlier version of Stone's theorem and  our versions  is that the spectrum of a self-adjoint generator is included in the real line, while that of an anti-self-adjoint generator is included in the 3-dimensional real vector space consisting of all pure imaginary quaternions.

 \begin{remark}
  The generalized Stone's theorems $($Theorems \ref{Thm-Stone-I} and \ref{Thm-Stone-II}$)$ can be  proved in almost the same way as the case when $\mathcal{H}$ is a complex Hilbert space $($see, e.g., Chap. 10 in \cite{Hall-2013}$)$.

  In fact, the following relation between a semigroup $U(t)$ on a quaternionic Hilbert space and  its infinitesimal generator $A$:
  $$U(t)=e^{tA}$$
  has been verified in several settings
  $($see, e.g., Theorem 3.2 in \cite{Colombo-2011} and Theorem 6.3 in \cite{Ghiloni-2018}$)$.
  Compared with the previous contributions made by others, Theorems \ref{Thm-Stone-I} and \ref{Thm-Stone-II} are more deeply involved with the functional calculus. And what matters in this paper is that they are vital for us to achieve our purpose. Only based on these two generalized Stone's Theorems, we are able to reveal the spectral characteristics of quaternionic  positive definite functions.
 \end{remark}

\begin{lemma}\label{Lemma-exp-anti-self-adjoint}
  Suppose $A$ is a right linear anti-self-adjoint operator on $\mathcal{H}$, and $U(t)$ is a family of operators defined as
  $$U(t)=e^{tA}\rvert_S,\ t\in\mathbb R,$$
  then the following results hold true:

  1. $U(t)$ is a strongly continuous one-parameter unitary group.

  2. For any $x\in \mathrm{D}(A)$,
  $$
  A(x)=\lim_{t\to 0}\frac{U(t)(x)-x}{t},
  $$
  where the limit is in the norm topology on $\mathcal{H}$.

  3. For any $x\in \mathcal{H}$, if
  $$\lim_{t\to 0}\frac{U(t)(x)-x}{t}
  $$
  exists in the norm topology on $\mathcal{H}$, then $x\in \mathrm{D}(A)$.

\end{lemma}

\begin{proof}
  Since $\sigma_S(A)$ contains only pure imaginary quaternions, the function $f_t(p):=e^{tp}$ is a bounded continuous intrinsic slice function on $\sigma_S(A)$. More precisely,
  $$f_t(jv)=\cos(tv)+j\sin(tv),$$
  holds for any $j\in\mathbb{S}$ and any $v\in\mathbb{R}$ with $jv\in\sigma_S(A)$.
  Hence, for different $j,j'\in\mathbb{S}$, $f_t(jv)$ and $f_t(j'v)$ share the same real and imaginary components that are both bounded and continuous. Therefore the functional calculus for $f(p)$ on the S-spectrum of $A$ is well defined as shown in Definition \ref{Def-FC-S-Spec}:
  \begin{equation*}
   \begin{split}
      \langle f_t(A)x,y\rangle
    = & \int_{\sigma_S(A)\cap \mathbb{C}_j^+}
      \cos(t\lvert p\rvert )d \langle E_j(p)x,y\rangle + \\
        & \int_{\sigma_S(A)\cap \mathbb{C}_j^+}
      \sin(t\lvert p\rvert )d\langle JE_j(p)x,y\rangle,
   \end{split}
  \end{equation*}
  where $(E_j,J)$ is the spectral system associated with  $A$.

  The functional calculus on the S-spectrum is also a $*$-homomorphism of real (not quaternionic) Banach $C^*$-algebras like the classical functional calculus \cite{Alpay-2016}. It indicates that
  $$U(t)U(t)^*=f_t(A)\overline{f_t}(A)=(f_t\overline{f_t})(A)=1(A)=I,$$
  $$U(t)^*U(t)=\overline{f_t}(A)f_t(A)=(\overline{f_t}f_t)(A)=1(A)=I,$$
  $$U(s)U(t)=f_s(A)f_t(A)=(f_sf_t)(A)=f_{s+t}(A)=U(s+t); $$
  which is to say $U(t)$ is a one-parameter unitary group. Moreover, for any $x\in \mathcal{H}$ and $s, t\in \mathbb{R}$, we have
  \begin{equation*}
   \begin{split}
   \|U(s)(x)-U(t)(x)\|^2
    =&\langle  (f_s(A)-f_t(A))^*(f_s(A)-f_t(A))x, x\rangle\\
    =&\langle   \lvert f_s-f_t \rvert^2(A)x, x\rangle.
   \end{split}
  \end{equation*}
  Definition \ref{Def-FC-S-Spec} yields
  $$
  \langle  \lvert f_s-f_t\rvert^2(A)x, x\rangle
    =\int_{\sigma_S(A)\cap \mathbb{C}_j^+}
     \big(\cos((s-t)\lvert p\rvert )-1\big)^2+\big(\sin((s-t)\lvert p\rvert )\big)^2 d\mu_{j,x}(p)
  $$
  with
  $\mu_{j,x}=\langle E_j(p)x,x\rangle$ being a finite Borel measure.
  The integral on the right side  tends to zero as $s$ approaches
  $t$, by dominated convergence. Thus we reach the first conclusion:

  1. $U(t)$ is a strongly continuous one-parameter unitary group.

  To see the second conclusion, first notice that Corollary 6.5 in \cite{Alpay-2016} indicates:
  \begin{eqnarray*}
     & &\Big\|\frac{U(t)(x)-x}{t}-A(x)\Big\|^2\\
     &=& \int_{\sigma_S(A)\cap \mathbb{C}_j^+}
      \lvert\frac{e^{tp}-1}{t}-p \rvert^2 d\mu_{j,x}(p) \\
     &=& \int_{\sigma_S(A)\cap \mathbb{C}_j^+}
      \lvert\frac{\cos(t\lvert p\rvert )-1}{t} \rvert^2 +  \lvert\frac{\sin(t\lvert p\rvert )}{t}-\lvert p\rvert  \rvert^2 d\mu_{j,x}(p)
  \end{eqnarray*}
  is valid
  for all $x\in \mathrm{D}(A)$ and all $t\in\mathbb R$.
  Then we apply the dominated convergence theorem again with $5\lvert p\rvert ^2$ as the dominating function to achieve the desired result:

  2.  For any $x\in \mathrm{D}(A)$,
  $$
  A(x)=\lim_{t\to 0}\frac{U(t)(x)-x}{t},
  $$
  where the limit is in the norm topology of $\mathcal{H}$.

  For the third conclusion, let $A'$ be the infinitesimal generator of $U(t)$. For any $x,y\in\mathrm{D}(A')$, one can easily see
  \begin{eqnarray*}
    \langle A'(x),y\rangle
    &=&\lim_{t\to 0}\langle \frac{U(t)(x)-x}{t},y\rangle \\
    &=& \lim_{t\to 0}\langle x,\frac{U(-t)(y)-y}{t}\rangle \\
    &=& \langle x,-A'(y)\rangle
  \end{eqnarray*}
  Hence, $A'$ is anti-symmetric. Combining with the second conclusion we have:

  (1) $A'$ is an extension of the anti-self-adjoint operator $A$,

  (2) $\mathrm{D}(A')\subset\mathrm{D}((A')^*)$,\\
 and consequently $\mathrm{D}(A)$ is identical with $\mathrm{D}(A')$, which indicates the third conclusion:

 3. For any $x\in \mathcal{H}$, if
  $$\lim_{t\to 0}\frac{U(t)(x)-x}{t}
  $$
  exists in the norm topology of $\mathcal{H}$, then $x\in \mathrm{D}(A)$.

\end{proof}

\begin{lemma}\label{Lemma-anti-self-adjoint}
  For any strongly continuous one-parameter unitary group $U(t)$ on $\mathcal{H}$, its infinitesimal generator $A$ is anti-self-adjoint.
\end{lemma}

\begin{proof}
   Set $\displaystyle{\langle x,y\rangle_j:=\frac{1}{2}\big(\langle x,y\rangle-j\langle x,y\rangle j\big)}$ with $j$ being an arbitrary imaginary unit. It can be easily seen that
 $\langle \cdot ,\cdot\rangle_j$ is a $\mathbb{C}_j$-linear inner product,
 and the norm induced by $\langle \cdot ,\cdot\rangle_j$ is identical with the one induced by $\langle
 \cdot ,\cdot\rangle$, which implies that  $\mathcal{H}$, as a $\mathbb{C}_j$-linear vector space, endowed with the inner product $\langle \cdot ,\cdot\rangle_j$ is a complex Hilbert space.
 Furthermore,
 the adjoint  of any densely defined right quaternion-linear operator with respect to the quaternionic inner product $\langle \cdot ,\cdot\rangle$ is identical with
 the adjoint  with respect to the complex inner product $\langle \cdot ,\cdot\rangle_j$.

 From this point of view, $U(t)$ can be treated as a strongly continuous one-parameter unitary group in the complex Hilbert space $(\mathcal{H},\langle \cdot ,\cdot\rangle_j)$, and $AR_j(=R_jA)$ is exactly the classical infinitesimal generator of $U(t)$ (see, e.g., Definition 10.13 in \cite{Hall-2013}), where $R_j$ is the right scalar multiplication by $j$, i.e.,
 $$R_j(x)=xj, \quad \forall \ x\in \mathcal{H}.
 $$
 We thus have $AR_j$ is self-adjoint  in the complex Hilbert space $(\mathcal{H},\langle \cdot ,\cdot\rangle_j)$ according to the original Stone's theorem (see, e.g., Theorem 10.15 in \cite{Hall-2013}), which indicates $A$ is anti-self-adjoint  in $(\mathcal{H},\langle \cdot ,\cdot\rangle_j)$. Since the adjoint  of $A$ with respect to $\langle \cdot ,\cdot\rangle$ is identical with
 the adjoint  with respect to $\langle \cdot ,\cdot\rangle_j$, we conclude that
 $A$ is anti-self-adjoint in the quaternionic Hilbert space $(\mathcal{H},\langle \cdot ,\cdot\rangle)$.

\end{proof}

\begin{proof}[Proof of Theorem \ref{Thm-Stone-I}]

  Suppose $U(t)$ is a strongly continuous one-parameter unitary group on $\mathcal{H}$, and $A$ is its infinitesimal generator. By Lemma \ref{Lemma-anti-self-adjoint}, $A$ is anti-self-adjoint. Then
  by Lemma \ref{Lemma-exp-anti-self-adjoint}, $e^{tA}\rvert_S$ is a strongly continuous one-parameter unitary group with  the infinitesimal generator identical with $A$.

  Take $x\in \mathrm{D}(A)$, and consider the function $g_x(t)=U(t)(x)-e^{tA}\rvert_S(x)$. From the definition of the infinitesimal generator, it follows immediately that
  $$\frac{d}{dt}U(t)(x)=AU(t)(x)=U(t)A(x),$$
  and
  $$\frac{d}{dt}e^{tA}\rvert_S(x)=Ae^{tA}\rvert_S(x)=e^{tA}\rvert_SA(x),$$
  hold in the norm topology of $\mathcal{H}$, which means
  $U(t)(x)$ and $e^{tA}\rvert_S(x)$ both belong to $\mathrm{D}(A)$, and
  $$\frac{d}{dt} g_x(t)=A\big(U(t)(x)-e^{tA}\rvert_S(x)\big)=Ag_x(t).
  $$
  Hence,
  $$
  \begin{array}{rcl}
   \displaystyle{\frac{d}{dt} \langle g_x(t),g_x(t)\rangle}
      & = &
       \displaystyle{\langle \frac{d}{dt}g_x(t),g_x(t)\rangle + \langle g_x(t),\frac{d}{dt}g_x(t)\rangle}\\
      & = & \langle Ag_x(t),g_x(t)\rangle + \langle g_x(t),Ag_x(t)\rangle\\
      & = & \langle g_x(t),-Ag_x(t)\rangle + \langle g_x(t),Ag_x(t)\rangle \\
      & = & 0.
  \end{array}
  $$
  Since $g_x(0)=0$, we  deduce that $g_x(t)=0$ holds for all $x\in\mathrm{D}(A)$ and all $t\in\mathbb R$, or equivalently
  $$U(t)(x)=e^{tA}\rvert_S(x), \quad \forall \ x\in\mathrm{D}(A),\ \forall \ t\in\mathbb R.
  $$
  In conclusion, $U(t)$ and $e^{tA}\rvert_S$ agree on a dense subspace of $\mathcal{H}$, and thus on the whole space.

\end{proof}

\begin{proof}[Proof of Theorem \ref{Thm-Stone-II}]

An analog of Lemma \ref{Lemma-exp-anti-self-adjoint} can be obtained by replacing $e^{tA}\rvert_S$ with $e^{tA}\rvert_L$. Then we can adopt the same procedure as in the proof of Theorem \ref{Thm-Stone-I} to carry out this one.  To avoid repetition, the details of this proof are omitted.

\end{proof}

\section{Bochner's theorem for quaternionic positive definite functions}\label{Sec-Bochner}
   In this section we are going to  show that every continuous quaternionic  positive definite function is related with a spectral system on $\mathbb{R}^+$ and also related with a non-negative finite Borel measure on $\mathbb R^3$ in  certain ways. Moreover, the spectral system and the Borel measure will induce two identical quaternion-valued measures of an unusual type that we name as slice-condensed measures. Finally, A one-to-one correspondence between continuous  quaternionic  positive definite functions and slice-condensed measures will be established.

   The earlier contribution \cite{Alpay-2016-1} by D. Alpay, F. Colombo, D. P. Kimsey, and I. Sabadini revealed a one-to-one correspondence between continuous  quaternionic  positive definite functions and q-positive measures. We would like to mention that slice-condensed measures are different with q-positive measures. And one may refer to \eqref{Eq-final-remark} in our final remark for a brief result about the connection between these two types of measure.

\begin{definition}\cite{Alpay-2016-1}
  A quaternion-valued function $\varphi$ on $\mathbb{R}$ is said to be  positive definite if for any $t_1,t_2,\cdots,t_k\in\mathbb{R}$  and any $q_1,q_2,\cdots,q_k\in\mathbb{H}$, the following inequality \begin{equation}\label{Eq-Def-PDF}
    \sum_{1\leq i,j\leq k} \overline{p_i}\varphi(t_i-t_j)p_j\geq0
  \end{equation}
  is satisfied.
\end{definition}

Before giving a  formal definition for slice-condensed measures,  we  introduce some notations.  $\mathbb{H}_I$ denotes the set of all pure imaginary quaternions, and $\mathbb{R}^+$ the set of non-negative real number.
$\mathscr{B}(X)$ stands for the Borel $\sigma$-algebra on a topological space $X$. The function $\rho: \mathbb{H}_I\mapsto\mathbb{R}^+$ is defined as
$$\rho(x):=\lvert x\rvert . $$
Evidently, $\rho$ is Borel measurable, then induces a push-forward mapping:
$$\rho_*(\Gamma)(\Omega):=\Gamma(\rho^{-1}(\Omega)),\ \forall \text{ Borel measure }\Gamma \text{ on }\mathbb{H}_I, \ \forall\ \Omega \in\mathscr{B}(\mathbb{R}^+).$$

\begin{definition}\label{Def-PNM-1}
  A quaternion-valued regular Borel measure $\mu$ on $\mathbb{R}^+$ is said to be slice-condensed if there exists a non-negative finite regular Borel measure $\Gamma$ on $\mathbb{H}_I$ such that the following equality holds:
  $$\mu=\rho_*(\Gamma+\frac{x}{\lvert x\rvert }\Gamma).$$
\end{definition}
Here we stipulate that $\frac{x}{\lvert x\rvert }=0$ when $x=0$, and consider this function as a Radon-Nikodym derivative, which means $\frac{x}{\lvert x\rvert }\Gamma$ is  a regular Borel measure defined by
$$\frac{x}{\lvert x\rvert }\Gamma(\Omega):=\int_{x\in\Omega}\frac{x}{\lvert x\rvert }d\Gamma(x)$$
for any Borel set $\Omega\in\mathscr{B}(\mathbb{H}_I)$.

This concept  can also be defined equivalently as follows.
\begin{definition}\label{Def-PNM-2}
  A quaternion-valued regular Borel measure $\mu$ on $\mathbb{R}^+$ is said to be slice-condensed if there exists a spectral system $(E:\mathscr{B}(\mathbb{R}^+)\mapsto\mathcal{B}(\mathcal{H}),J)$ and a point $\alpha$ in a quaternionic Hilbert space $\mathcal{H}$ such that the following equality holds:
  $$\mu=\langle E\alpha,\alpha\rangle + \langle J_0E\alpha,\alpha\rangle,$$
  where $J_0=J-JE(\{0\})$ and $\langle\cdot,\cdot\rangle$ stands for the inner product on $\mathcal{H}$.
\end{definition}
More precisely, this equality $\mu=\langle E\alpha,\alpha\rangle + \langle J_0E\alpha,\alpha\rangle$ means
$$\mu(\omega)=\langle E(\omega)\alpha,\alpha\rangle + \langle J_0E(\omega)\alpha,\alpha\rangle$$
is satisfied for any $\omega\in\mathscr{B}(\mathbb{R}^+)$.

\begin{remark}
  We would like to emphasize that Definitions \ref{Def-PNM-1} and \ref{Def-PNM-2} are equivalent, and this assertion will be illuminated in Subsection \ref{Subs-Equa-Def-PNM}. To prevent confusion, one may ignore Definition \ref{Def-PNM-1} temporarily until  reach Subsection \ref{Subs-Equa-Def-PNM}.
\end{remark}

Let $\mathcal{M}_{S}(\mathbb{R}^+)$ denote the set of all slice-condensed regular Borel measures on $\mathbb{R}^+$.
The next theorem will show there exists a one to one correspondence between the continuous quaternionic positive definite functions and the slice-condensed regular Borel measures.
\begin{theorem}[Generalized Bochner's theorem]\label{Bochner-theorem}
If a quaternion-valued function $\varphi$ on $\mathbb{R}$ is continuous and positive definite, then there exists a unique $\mu\in\mathcal{M}_{S}(\mathbb{R}^+)$ such that
$$\varphi(t)=\int_{\mathbb{R}^+}\cos(tx)d\mathbf{Re}\mu(x)+\int_{\mathbb{R}^+}\sin(tx)d(\mu-\mathbf{Re}\mu)(x), $$
and vice versa.
\end{theorem}

\subsection{The Hilbert space associated with a positive definite function}\label{H-space-PDF}

Let $\varphi$ be a quaternionic positive definite function, and $F_0(\mathbb{R},\mathbb{H})$ be the family of quaternion valued functions on $\mathbb{R}$ with finite support. Evidently, $F_0(\mathbb{R},\mathbb{H})$ has a natural right $\mathbb{H}$-linear structure that makes it a right $\mathbb{H}$-vector space. The positive definite function $\varphi$ will induce a (possibly degenerate) inner product:
 $$\langle f,g\rangle:=\sum_{s,t\in\mathbb{R}}\overline{g(s)}\varphi(s-t)f(t), $$
for all $f,g\in F_0(\mathbb{R},\mathbb{H})$.

Quotienting  $F_0(\mathbb{R},\mathbb{H})$ by the subspace of functions with zero norm  eliminates the degeneracy. Then taking the completion gives a quaternionic Hilbert space $(\mathcal{H}_\varphi,\langle \cdot,\cdot\rangle)$.

Note that if $\varphi$ vanishes at the origin, it can be easily seen that $\mathcal{H}$ is a 0-dimensional space, and $\varphi\equiv 0$; then all the main results are trivially true. So without loss of generality, we can always assume $$\varphi(0)\neq 0. $$

Recall that every quaternionic Hilbert space can be transformed into a bilateral Hilbert space as follows:

1. Introduce a new inner product $(\cdot\mid \cdot)$, given by
$$(x\mid y)=\overline{\langle x,y\rangle}, \quad \forall x, y\in\mathcal{H}_\varphi.$$

2. Construct a left $\mathbb{H}$-linear structure on $\mathcal{H}_\varphi$: Let $\{x_a\}$ be an orthonormal basis, the left scalar multiplication is defined as
$$qx=\sum_{a}x_aq\langle x,x_a\rangle, \quad \forall q\in \mathbb{H}, \ x\in \mathcal{H}_\varphi. $$
Then the $\mathbb{H}$-vector space $\mathcal{H}_\varphi$ endowed with the inner product $(\cdot\mid \cdot)$ can be treated as a bilateral quaternionic Hilbert space.
For convenience, we denote the bilateral quaternionic Hilbert space $(\mathcal{H}_\varphi,(\cdot\mid \cdot))$ by $\tilde{\mathcal{H}}_\varphi$.

In this very case, we choose the orthonormal basis specifically given as
 $$\{x_a\}:=\{\delta / \|\delta\|\}\cup\{x_\beta\},$$
where $\delta$ is the finite delta function given by
$$\delta(x)=\left\{\begin{array}{ccc}
                     1,& & x=0; \\
                     0,& & x\neq 0.
                   \end{array}\right.$$ and  $\{x_\beta\}$ is an arbitrary orthonormal basis of $\{\delta / \|\delta\|\}^{\perp}$.
Such choice ensures the following commutativity:
\begin{equation}\label{Eq-Commute-delta}
q\delta=\delta q, \quad \forall \ q\in\mathbb{H}.
\end{equation}

\begin{remark}
Note that the left scaler multiplication on $\tilde{\mathcal{H}}_\varphi$ is different from the conventional left scaler multiplication. In other words, for a quaternion-valued function $f$ on $\mathbb{R}$ with finite support and an arbitrary element $q\in\mathbb{H}$, the following equality
$$(qf)(x)=qf(x), \qquad x\in\mathbb{R},$$
is generally non-valid in $\tilde{\mathcal{H}}_\varphi$. So the commutativity in \eqref{Eq-Commute-delta} is not trivial.
\end{remark}

\subsection{Spectral theorems for quaternionic positive definite functions}\label{Subs-Spectral-Thm}

\begin{theorem}\label{Thm-Spectral}
A quaternion-valued function $\varphi$ defined on $\mathbb{R}$  is continuous and positive definite if and only if there exist  a spectral system $(E:\mathscr{B}(\mathbb{R}^+)\mapsto\mathcal{B}(\mathcal{H}),J)$ and a point $\alpha$ in a quaternionic Hilbert space $\mathcal{H}$ such that
$$\varphi(t)=\int_{\mathbb{R}^+}\cos(tx)\ d\langle E\alpha,\alpha\rangle(x)+\int_{\mathbb{R}^+}\sin(tx)\ d\langle JE\alpha,\alpha\rangle(x). $$
\end{theorem}

\begin{proof}

  First, we shall show the sufficiency. Assume that for a function $\varphi$, there exist  a spectral system $(E:\mathscr{B}(\mathbb{R}^+)\mapsto\mathcal{B}(\mathcal{H}),J)$ and a point $\alpha$ in a quaternionic Hilbert space $\mathcal{H}$ such that
$$\varphi(t)=\int_{\mathbb{R}^+}\cos(tx)\ d\langle E\alpha,\alpha\rangle(x)+\int_{\mathbb{R}^+}\sin(tx)\ d\langle JE\alpha,\alpha\rangle(x). $$
Obviously, $\varphi$ is continuous in view of the dominated convergence theorem.
We only need to check whether $\varphi$ is positive definite.

By Lemma 5.3 of \cite{Alpay-2016}, $\varphi(t)$ is identical with $\langle\mathbb{I}(f_t)\alpha,\alpha\rangle$ where $f_t(x):=e^{tx}$ and $\mathbb{I}$ is a $*$-homomorphism induced by the spectral system $(E,J)$.
For arbitrary $t_1,t_2,\cdots,t_k\in\mathbb{R}$  and  $q_1,q_2,\cdots,q_k\in\mathbb{H}$, we have
\begin{eqnarray*}
  \sum_{1\leq i,j\leq k} \overline{p_i}\varphi(t_i-t_j)p_j
   &=&\sum_{1\leq i,j\leq k} \overline{p_i}\langle\mathbb{I}(f_{t_i-t_j})\alpha,\alpha\rangle p_j \\
   &=& \sum_{1\leq i,j\leq k} \overline{p_i}\langle\mathbb{I}(f_{t_i})\mathbb{I}( f_{-t_j})\alpha,\alpha\rangle p_j\\
   &=& \sum_{1\leq i,j\leq k} \overline{p_i}\langle\mathbb{I}( f_{-t_j})\alpha,\mathbb{I}(f_{-t_i})\alpha\rangle p_j
\end{eqnarray*}
Subsequently, since $\langle xp,yq\rangle=\overline q \langle x,y\rangle p$ for all $p,q\in\mathbb{H}$ and all $x,y\in\mathcal{H}$, the following equality holds true:
\begin{eqnarray*}
  \sum_{1\leq i,j\leq k} \overline{p_i}\varphi(t_i-t_j)p_j
   &=& \sum_{1\leq i,j\leq k} \langle\mathbb{I}( f_{-t_j})\alpha p_j,\mathbb{I}(f_{-t_i})\alpha p_i\rangle  \\
   &=& \| \sum_{1\leq j\leq k}\mathbb{I}( f_{-t_j})\alpha p_j\|^2  \\
   &\geq&0
\end{eqnarray*}
Hence $\varphi$ is positive definite.

Next, we shall verify the necessity. The quaternionic Hilbert space $\mathcal{H}_\varphi$ constructed in Subsection \ref{H-space-PDF} will come into immediate use.

Assume that $\varphi:\mathbb{R}\mapsto\mathbb{H}$ is a continuous positive definite function.
Consider a family of shift operators $U_t$ $(t\in\mathbb{R})$ on $F_0(\mathbb{R},\mathbb{H})$ given by
\begin{equation}\label{Eq-Def-U_t}
  U_t(f)=f(\cdot+t).
\end{equation}
The following facts hold:

1. $U_t$ is a right $\mathbb{H}$-linear bijection preserving the (possibly degenerate) inner product $\langle\cdot,\cdot
\rangle$ induced by the positive definite function $\varphi$ for all $t\in\mathbb{R}$.

2. $U_0=I$ and $U_tU_s=U_{t+s}$ for all $t,s\in\mathbb{R}$.

3. The continuity of $\varphi$ implies that $\displaystyle{\lim_{s\to t}}\| U_s(f)-U_t(f)\|=0$ holds for all $f\in F_0(\mathbb{R},\mathbb{H})$ and all $t\in\mathbb{R}$.

Thus $U_t$ $(t\in\mathbb{R})$ can be uniquely extended  as a strongly continuous one-parameter unitary group on $\mathcal{H}_\varphi$. It follows directly from Theorem \ref{Thm-Stone-I} that
the infinitesimal generator, denoted by $A$, of $U_t$ is a right linear anti-self-adjoint operator, and
 \begin{equation}\label{Eq-Thm-Spectral-1}
   U_t=e^{tA}\rvert_S \ \text{for all} \ t\in \mathbb{R}.
 \end{equation}

Here $e^{tA}\rvert_S$ is defined by the functional calculus for the intrinsic slice function $e^{tx}$ on the S-spectrum of $A$. More precisely,
 \begin{equation}\label{Eq-Thm-Spectral-2}
 \begin{split}
    \langle e^{tA}\rvert_Sx,y\rangle
    = &  \int_{\sigma_S(A)\cap \mathbb{C}_j^+}
      \mathbf{Re}(e^{tp})d \langle E_jx,y\rangle (p)+\\
      & \int_{\sigma_S(A)\cap \mathbb{C}_j^+}
      \mathbf{Im}(e^{tp})d\langle JE_jx,y\rangle (p),
 \end{split}
 \end{equation}
where $(E_j:\mathscr{B}(\sigma_S(A)\cap \mathbb{C}_j^+)\mapsto\mathcal{B}(\mathcal{H}_\varphi),J)$ is the spectral system  associated with $A$. Moreover,
$$\mathbf{Re}(e^{tp})=\cos(t\lvert p\rvert )
  \ \text{ and }\
  \mathbf{Im}(e^{tp})=\sin(t\lvert p\rvert ),
$$
since  $\sigma_S(A)$ contains  only pure imaginary quaternions.

Define a resolution of identity $E:\mathscr{B}(\mathbb{R}^+)\mapsto\mathcal{B}(\mathcal{H}_\varphi)$ as
$$E(\omega)=E_j(\sigma_S(A)\cap \mathbb{C}_j^+ \cap j\omega), \qquad \omega\in\mathscr{B}(\mathbb{R}^+). $$
One may notice that $\sigma_S(A)\cap \mathbb{C}_j^+$ is in fact a subset of $j\mathbb{R}^+$. So $E$ is essentially a zero extension of $E_j$. Furthermore,  \eqref{Eq-Thm-Spectral-1} and \eqref{Eq-Thm-Spectral-2} indicate
\begin{equation}\label{Eq-Thm-Spectral-3}
  \langle U_t\delta,\delta \rangle =\int_{\mathbb{R}^+}
      \cos (tx) d \langle E\delta,\delta \rangle (x)+
     \int_{\mathbb{R}^+}
      \sin (tx) d\langle JE\delta,\delta\rangle (x),
\end{equation}
      where $\delta$ is the finite delta function given by
$$\delta(x)=\left\{\begin{array}{ccc}
                     1,& & x=0; \\
                     0,& & x\neq 0.
                   \end{array}\right.$$
A direct calculation yields the left side of \eqref{Eq-Thm-Spectral-3} is also equal to $\varphi(t)$. This completes the proof.
\end{proof}

\begin{corollary}\label{Cor-Spectral}
 A quaternion-valued function $\varphi$ defined on $\mathbb{R}$  is continuous and positive definite if and only if there exist  a spectral system $(E:\mathscr{B}(\mathbb{R}^+)\mapsto\mathcal{B}(\mathcal{H}),J)$ and a point $\alpha$ in a quaternionic Hilbert space $\mathcal{H}$ such that
$$\varphi(t)=\int_{\mathbb{R}^+}\cos(tx)\ d\langle E\alpha,\alpha\rangle(x)+\int_{\mathbb{R}^+}\sin(tx)\ d\langle J_0E\alpha,\alpha\rangle(x), $$
where $J_0=J-JE(\{0\})$.
\end{corollary}
\begin{proof}
  For any $\omega\in\mathscr{B}(\mathbb{R}^+)$ with $0\not\in\omega$, we have
  $$E(\{0\})E(\omega)=E(\{0\}\cap\omega)=E(\emptyset)=0. $$
  Hence $JE(\omega)=J_0E(\omega)$, which means $\langle JE(\omega)\alpha,\alpha\rangle=\langle J_0E(\omega)\alpha,\alpha\rangle$.
  Then we obtain
  $$\int_{\mathbb{R}^+}\sin(tx)\ d\langle JE\alpha,\alpha\rangle(x)=\int_{\mathbb{R}^+}\sin(tx)\ d\langle J_0E\alpha,\alpha\rangle(x),$$
  since the integrand $\sin(tx)$ vanishes at $x=0$ and these two measures $\langle JE\alpha,\alpha\rangle$ and $\langle J_0E\alpha,\alpha\rangle$ are identical on $(0,+\infty)$.
  Therefore, this corollary follows directly from Theorem \ref{Thm-Spectral}.
\end{proof}

\begin{proof}[Proof of Theorem \ref{Bochner-theorem}]

We only  verify the necessity. The sufficiency can be proved reversely.

Assume $\varphi:\mathbb{R}\mapsto\mathbb{H}$ is a continuous positive definite function. By Corollary \ref{Cor-Spectral}, there exist a spectral system $(E,J)$ and a point $\alpha$ in a quaternionic Hilbert space $\mathcal{H}$ such that
\begin{equation}\label{Eq-Thm-Bochner-1}
  \varphi(t)=\int_{\mathbb{R}^+}\cos(tx)\ d\langle E\alpha,\alpha\rangle(x)+\int_{\mathbb{R}^+}\sin(tx)\ d\langle J_0E\alpha,\alpha\rangle(x).
\end{equation}
Consider a regular Borel measure given by
$$\mu=\langle E\alpha,\alpha\rangle+\langle J_0E\alpha,\alpha\rangle. $$
By Definition \ref{Def-PNM-2}, i.e., the second definition of a slice-condensed measure, we know $\mu$ is slice-condensed.

Notice two facts:

1. $E(\omega)$ is self-adjoint for all $\omega\in\mathscr{B}(\mathbb{R}^+)$.

2. $J$  is anti-self-adjoint, and commutes with $E$.

Thus, $\langle E\alpha,\alpha\rangle$ is pure real valued and $\langle J_0E\alpha,\alpha\rangle$ is pure imaginary valued. This implies
$$\mathbf{Re}\mu=\langle E\alpha,\alpha\rangle.$$
Hence, \eqref{Eq-Thm-Bochner-1} indicates
$$\varphi(t)=\int_{\mathbb{R}^+}\cos(tx)d\mathbf{Re}\mu(x)+\int_{\mathbb{R}^+}\sin(tx)d(\mu-\mathbf{Re}\mu)(x).  $$

In addition, the uniqueness of the slice-condensed measure $\mu$ is a direct result of the Stone-Weierstrass theorem.

\end{proof}

\subsection{Equivalence between the two definitions of slice-condensed measures}\label{Subs-Equa-Def-PNM}
We claimed Definitions \ref{Def-PNM-1} and \ref{Def-PNM-2} are equivalent. Now this assertion will be verified.

First we recall some notations.
$\mathbb{H}_I$ denotes the set of pure imaginary quaternions, and $\mathbb{R}^+$ the set of non-negative real number. $\mathscr{B}(X)$ stands for the Borel $\sigma$-algebra on a topological space $X$. The function $\rho: \mathbb{H}_I\mapsto\mathbb{R}^+$ is defined as
$$\rho(x):=\lvert x\rvert . $$
It induces a push-forward mapping $\rho_*$ given by
\begin{equation}\label{Eq-Def-push-forward}
  \rho_*(\Gamma)(\Omega):=\Gamma(\rho^{-1}(\Omega)),
\end{equation}
for any Borel measure $\Gamma$ on $\mathbb{H}_I$, and any  $\Omega\in\mathscr{B}(\mathbb{R}^+)$.

Recall Definition \ref{Def-PNM-1}:

A quaternion-valued regular Borel measure $\mu$ on $\mathbb{R}^+$ is said to be slice-condensed if there exists a non-negative finite regular Borel measure $\Gamma$ on $\mathbb{H}_I$ such that the following equality holds:
  $$\mu=\rho_*(\Gamma+\frac{x}{\lvert x\rvert }\Gamma).$$
Here we stipulate that $\frac{x}{\lvert x\rvert }=0$ when $x=0$, and consider this function as a Radon-Nikodym derivative, which means $\frac{x}{\lvert x\rvert }\Gamma$ is  a regular Borel measure defined by
$$\frac{x}{\lvert x\rvert }\Gamma(\Omega):=\int_{x\in\Omega}\frac{x}{\lvert x\rvert }d\Gamma(x)$$
for any Borel set $\Omega\in\mathscr{B}(\mathbb{H}_I)$.

Recall Definition \ref{Def-PNM-2}:

A quaternion-valued regular Borel measure $\mu$ on $\mathbb{R}^+$ is said to be slice-condensed if
there exists a spectral system $(E:\mathscr{B}(\mathbb{R}^+)\mapsto\mathcal{B}(\mathcal{H}),J)$ and a point $\alpha$ in a quaternionic Hilbert space $\mathcal{H}$ such that the following equality holds:
  $$\mu=\langle E\alpha,\alpha\rangle + \langle J_0E\alpha,\alpha\rangle,$$
  where $J_0$ is given by $J_0=J-JE(\{0\})$ and $\langle\cdot,\cdot\rangle$ stands for the inner product on $\mathcal{H}$.

For ease of explanation, we call any measure that satisfies Definition \ref{Def-PNM-1} is a slice-condensed measure of type I, any measure that satisfies Definition \ref{Def-PNM-2} is a slice-condensed measure of type II.
\begin{lemma}\label{Lemma-Eq-PNM-1}
  Any slice-condensed measure of type I is a slice-condensed measure of type II.
\end{lemma}

\begin{proof}
Assume $\mu$ is a slice-condensed measure of type I, then there exists a non-negative finite regular Borel measure $\Gamma$ on $\mathbb{H}_I$ such that the following equality holds:
  \begin{equation}\label{Eq-mu}\mu=\rho_*(\Gamma+\frac{x}{\lvert x\rvert }\Gamma)
  \end{equation}

Let $L^2(\mu,\mathbb{H})$ denote the quaternionic Hilbert space consisting of  Borel measurable quaternion-valued functions on $\mathbb{H}_I$ which are square-integrable with respect to the measure $\Gamma$. The inner product on $L^2(\mu,\mathbb{H})$ is naturally given by
$$
\langle f,g\rangle:=\int_{\mathbb{H}_I}\overline{g(x)} f(x) d\Gamma(x).
$$

Consider a resolution of identity $E'$ on $\mathbb{H}_I$ and an anti-self-adjoint unitary operator $J$ defined as follows:

$$E'(\omega)f(x):=\chi_\omega(x) f(x),$$
for all $\omega\in \mathscr{B}(\mathbb{H}_I)$, and all $f\in L^2(\mu,\mathbb{H})$. Here $\chi_\omega$ is the characteristic function of  the set $\omega$.

$$Jf(x):=\left\{\begin{array}{lll}
                 \frac{x}{\lvert x\rvert }f(x),& &x\neq 0;  \\
                 jf(0),& & x=0;
               \end{array}\right. $$
where $j$ is an arbitrary imaginary unit.
It can be easily seen that $E'$ commutes with $J$.

Applying the push-forward mapping $\rho_*$ given by \eqref{Eq-Def-push-forward} to $E'$ produces a resolution of identity $E$ on $\mathbb{R}^+$ as
$$\langle E f,g\rangle:=\rho_*(\langle E'f,g\rangle),\qquad \forall f,\ g\in L^2(\mu,\mathbb{H}). $$
Equivalently, the resolution of identity $E$ is defined by
$$E(\omega):=E'(\rho^{-1}(\omega)), \qquad \forall \omega\in \mathscr{B}(\mathbb{R}^+),$$
where $\rho(x)=\lvert x\rvert $, $x\in\mathbb{H}_I$.
We notice that $E$ also commutes with $J$. Thus $(E,J)$ is a spectral system according to Definition \ref{Def-FCS}.  Furthermore, direct calculations yield
$$\langle E(\omega)\alpha,\alpha\rangle=\rho_*\Gamma(\omega),$$
and
$$\langle J_0E(\omega)\alpha,\alpha\rangle=\rho_*(\frac{x}{\lvert x\rvert }\Gamma)(\omega),$$
for all $\omega\in \mathscr{B}(\mathbb{R}^+)$,
where $\alpha$ is given as the characteristic function of $\mathbb{H}_I$, and $J_0=J-JE(\{0\})$.
Substituting the two equalities above into \eqref{Eq-mu}, we thus obtain
 $$\mu=\langle E\alpha,\alpha\rangle + \langle J_0E\alpha,\alpha\rangle. $$
Therefore, $\mu$ is a slice-condensed measure of type II.
\end{proof}

\begin{lemma}\label{Lemma-Eq-PNM-2}
  Any slice-condensed measure of type II is a slice-condensed measure of type I.
\end{lemma}
\begin{proof}
This proof is lengthy. We would like to outline it, then give the details.

Let $\mu$ be a slice-condensed measure of type II. Corollary \ref{Cor-Spectral} indicates that the following  function
\begin{equation}\label{EQ-OULINE-varphi}
 \varphi(t):=\int_{\mathbb{R}^+}\cos(tx)d\mathbf{Re}\mu(x)+\int_{\mathbb{R}^+}\sin(tx)d(\mu-\mathbf{Re}\mu)(x),
\end{equation}
is continuous and positive definite. As shown in Subsection \ref{H-space-PDF}, there is a quaternionic Hilbert apace $\mathcal{H}_\varphi$ associated with $\varphi$.
Applying Theorem \ref{Thm-Stone-II} to the unitary group $U_t$ on $\mathcal{H}_\varphi$ given by \eqref{Eq-Def-U_t}, we obtain
  $$U_t=e^{tA}\rvert_L $$
with $A$ being the infinitesimal generator of $U(t)$.
Definition \ref{Def-FC-L-Spec}, along with the commutativity shown in \eqref{Eq-Commute-delta},  yields two facts:\\
1.
\begin{equation*}
\begin{split}
   (U_t \delta \mid \delta )=& \sum_{l=0}^{3}\int_{\sigma_L(A)}\cos t\lvert x\rvert  i_ld\mu_{i_0,i_l}[\delta,\delta](x) - \\
     & \sum_{k=1}^{3}\sum_{l=0}^{3}\int_{\sigma_L(A)}\frac{x_k}{\lvert x\rvert }\sin t\lvert x\rvert i_li_kd\mu_{i_0,i_l}[\delta,\delta](x),
\end{split}
\end{equation*}
where  $\delta$ is the finite delta function given by
$$\delta(x)=\left\{\begin{array}{ccc}
                     1,& & x=0; \\
                     0,& & x\neq 0;
                   \end{array}\right.$$
and $\mu_{i_0,i_l}[\delta,\delta]$ $(l=0,1,2,3)$ are regular Borel measures on $\sigma_L(A)$ determined by $A$ and $\delta$.\\
2.  $$\mu_{i_0,i_l}[\delta,\delta]\ \left\{\begin{array}{lcl}
                                             \text{is finite and non-negative,} &  & l=0, \\
                                             &  & \\
                                             =0 &  & l=1,2,3.\end{array} \right. $$
We thus have
\begin{equation}\label{Eq-OUTLINE-1}
   (U_t\delta\mid \delta)= \int_{\mathbb{H}_I}\cos t\lvert x\rvert d\Gamma(x)+ \int_{\mathbb{H}_I}\frac{x}{\lvert x\rvert }\sin t\lvert x\rvert d\Gamma(x),
\end{equation}
where  $\Gamma$ is a non-negative finite regular Borel measure  on $\mathbb{H}_I$, the set of pure imaginary quaternions, defined by
$$\Gamma(\omega):=\mu_{i_0,i_0}[\delta,\delta]\big(\sigma_L(A)\cap(-\omega)\big), \quad \forall\ \omega\in\mathscr{B}(\mathbb{H}_I). $$
By the definitions of $U_t$ and $\delta$, it is easy to see that the left side of \eqref{Eq-OUTLINE-1} equals $\varphi(t)$.
On the other hand,
a direct calculation yields the right side of \eqref{Eq-OUTLINE-1} equals  $$\int_{\mathbb{R}^+}\cos tx \ d\mathbf{Re}\mu'(x)+\int_{\mathbb{R}^+}\sin tx\  d(\mu'-\mathbf{Re}\mu')(x)$$
where $\mu'$ is a slice-condensed measure of type I given by $$\mu':=\rho_*(\Gamma+\frac{x}{\lvert x\rvert }\Gamma). $$
Hence, $$\varphi(t)=\int_{\mathbb{R}^+}\cos tx \ d\mathbf{Re}\mu'(x)+\int_{\mathbb{R}^+}\sin tx\  d(\mu'-\mathbf{Re}\mu')(x)$$
Comparing this equality with \eqref{EQ-OULINE-varphi} we conclude that the  slice-condensed measure $\mu$ of type II is identical with the slice-condensed measure $\mu'$ of type I. It completes the proof.

The details are as follows:

{\bf Step 1:}

Assume $\mu$ is a slice-condensed measure of type II. Then there exists a spectral system $(E:\mathscr{B}(\mathbb{R}^+)\mapsto\mathcal{B}(\mathcal{H}),J)$ and a point $\alpha$ in a quaternionic Hilbert space $\mathcal{H}$ such that the following equality holds:
  $$\mu=\langle E\alpha,\alpha\rangle + \langle J_0E\alpha,\alpha\rangle,$$
where $J_0=J-JE(\{0\})$. It can be seen easily that $\mathbf{Re}\mu(x)=\langle E\alpha,\alpha\rangle$, since $E$ is self-adjoint and $J_0$ is anti-self-adjoint.

Corollary \ref{Cor-Spectral} yields that  the function $\varphi$ given by  \begin{equation}\label{Eq-Def-varphi}
  \varphi(t):=\int_{\mathbb{R}^+}\cos(tx)d\mathbf{Re}\mu(x)+\int_{\mathbb{R}^+}\sin(tx)d(\mu-\mathbf{Re}\mu)(x),
\end{equation}
is continuous and positive definite.
As shown in Subsection \ref{H-space-PDF}, there is a quaternionic Hilbert apace $\mathcal{H}_\varphi$ associated with $\varphi$; moreover, $\mathcal{H}_\varphi$ can be transformed into a bilateral quaternionic Hilbert apace $\tilde{\mathcal{H}}_\varphi$.

In the proof of \ref{Thm-Spectral}, it has been illuminated that the family of operators $U_t$ $(t\in\mathbb{R})$ defined by \eqref{Eq-Def-U_t} is a strongly continuous one-parameter unitary group on $\mathcal{H}_\varphi$. It follows immediately from Theorem \ref{Thm-Stone-II} that
  $$U_t=e^{tA}\rvert_L \ \text{for all} \ t\in \mathbb{R}, $$
Here, $A$ is the infinitesimal generator of $U(t)$, and $e^{tA}\rvert_L$ is defined by the functional calculus for the function $e^{tx}$ on the left spectrum of $A$ in the bilateral Hilbert space $\tilde{\mathcal{H}}_\varphi$.

{\bf Step 2:}

Because $A$ is anti-self-adjoint, the left spectrum $\sigma_L(A)$ must be a subset of the pure imaginary space $\mathbb{H}_I:= i_1\mathbb{R}+ i_2\mathbb{R} + i_2\mathbb{R}$. Due to this fact, we know the function $e^{tx}$ can be composed as $$e^{tx}=\cos t\lvert x\rvert +i_1\frac{x_1}{\lvert x\rvert }\sin t\lvert x\rvert +i_2\frac{x_2}{\lvert x\rvert }\sin t\lvert x\rvert +i_3\frac{x_3}{\lvert x\rvert }\sin t\lvert x\rvert ,$$
for all $x=x_0+x_1i_1+x_2i_2+x_3i_3$ in the left spectrum. Definition \ref{Def-FC-L-Spec} thus leads to the following equality:
\begin{equation}\label{Eq-exp-L-1-1}
  (e^{tA}\rvert_L \delta \mid \delta )=\big(\phi(\cos t\lvert x\rvert )\delta\mid \delta\big)+\sum_{k=1}^{3}\big(i_k\phi(\frac{x_k}{\lvert x\rvert }\sin t\lvert x\rvert )\delta\mid \delta\big),
\end{equation}
where $\delta$ is the finite delta function given by
$$\delta(x)=\left\{\begin{array}{ccc}
                     1,& & x=0; \\
                     0,& & x\neq 0;
                   \end{array}\right.$$
since $\phi$ is a $\mathbb{H}$-algebra homomorphism from $\mathcal{B}(\sigma_L(A),\mathbb{H})$ to $\mathcal{B}_q(\tilde{\mathcal{H}}_\varphi)$.
Moreover, the definition of left scalar multiplication (one may refer to Subsection \ref{H-space-PDF} for more details) indicates that $L_{i_k}$, i.e., the left scalar multiplication by $i_k$ ($k=1,2,3$), is an anti-self-adjoint right $\mathbb{H}$-linear bounded operator. Hence, by shifting the position of $i_k$ in \eqref{Eq-exp-L-1-1} we have
$$(e^{tA}\rvert_L \delta \mid \delta )=\big(\phi(\cos t\lvert x\rvert )\delta\mid \delta\big)-\sum_{k=1}^{3}\big(\phi(\frac{x_k}{\lvert x\rvert }\sin t\lvert x\rvert )\delta\mid i_k\delta\big),$$
which, along with \eqref{Eq-Commute-delta}, implies
\begin{equation}\label{Eq-equivalence-PNM-exp-tA-1}
  (e^{tA}\rvert_L \delta \mid \delta )=\big(\phi(\cos t\lvert x\rvert )\delta\mid \delta\big)-\sum_{k=1}^{3}\big(\phi(\frac{x_k}{\lvert x\rvert }\sin t\lvert x\rvert )\delta\mid \delta\big)i_k,
\end{equation}

By Definition \ref{Def-FC-L-Spec} there exist regular $\mathbb{R}$-valued Borel measure $\mu_{i_v,i_l}[\alpha,\beta]$ $(v,l=0,1,2,3; \alpha,\beta\in\tilde{\mathcal{H}}_\varphi)$  such that
\begin{equation}\label{Eq-mu-expression}
(\phi(f)\alpha\mid \beta)=\sum_{v,l=0}^{3}\int_{\sigma_L(A)}f_vi_l\ d\mu_{i_v,i_l}[\alpha,\beta],
\end{equation}
where $f=\sum_{v=0}^{3}f_vi_v\in\mathcal{B}(\sigma_L(A),\mathbb{H})$ and $f_v$ is $\mathbb{R}$-valued. Applying \eqref{Eq-mu-expression}  to \eqref{Eq-equivalence-PNM-exp-tA-1} yields
\begin{equation}\label{Eq-equivalence-PNM-exp-tA-2}
\begin{split}
   (e^{tA}\rvert_L \delta \mid \delta )=& \sum_{l=0}^{3}\int_{\sigma_L(A)}\cos t\lvert x\rvert  i_ld\mu_{i_0,i_l}[\delta,\delta](x) \ - \\
     & \sum_{k=1}^{3}\sum_{l=0}^{3}\int_{\sigma_L(A)}\frac{x_k}{\lvert x\rvert }\sin t\lvert x\rvert i_li_kd\mu_{i_0,i_l}[\delta,\delta](x).
\end{split}
\end{equation}

{\bf Step 3:}

Since $\phi$ is an involution preserving $\mathbb{H}$-algebra homomorphism, any $\mathbb{R}$-valued bounded measurable function $f$ on $\sigma_L(A)$ satisfies the following equations:

\begin{equation*}
\begin{split}
(\phi(f)\delta\mid -i_k\delta)&=(i_k\phi(f)\delta\mid \delta) \\
                                      &=(\phi(i_kf)\delta\mid \delta)\\
                                      \\
(\phi(f)\delta\mid i_k\delta)&=(-i_k\phi(f)\delta\mid \delta)\\
                         &=(\phi(\overline{i_kf})\delta\mid \delta)\\
                         \\
\mathbf{Re}(\phi(i_kf)\delta\mid \delta)&=\mathbf{Re}(\delta\mid \phi(\overline{i_kf})\delta)\\
                                    &=\mathbf{Re}(\phi(\overline{i_kf})\delta\mid \delta)
\end{split}\qquad k=1,2,3.
\end{equation*}
It thus follows that
\begin{equation*}
\mathbf{Re}(\phi(f)\delta\mid -i_k\delta)=\mathbf{Re}(\phi(f)\delta\mid i_k\delta),\qquad k=1,2,3.
\end{equation*}
holds true. Applying \eqref{Eq-Commute-delta} to the equality above leads to
\begin{equation}\label{Eq-mu-v-l-vanish-1}
\mathbf{Re}\Big((\phi(f)\delta\mid \delta)(-i_k)\Big)=\mathbf{Re}\Big((\phi(f)\delta\mid \delta)i_k\Big),\qquad k=1,2,3.
\end{equation}
We combine \eqref{Eq-mu-expression} and \eqref{Eq-mu-v-l-vanish-1} to deduce that
$$\int_{\sigma_L(A)}f\ d\mu_{i_0,i_k}[\delta,\delta]=-\int_{\sigma_L(A)}f\ d\mu_{i_0,i_k}[\delta,\delta], \qquad k=1,2,3,$$
is valid for any $\mathbb{R}$-valued bounded measurable function $f$ on $\sigma_L(A)$.
Due to the randomness of $f$, we come to the vital result that
$$\mu_{i_0,i_k}[\delta,\delta]=0, \qquad k=1,2,3.$$
Moreover, the following equality $$\int_{\sigma_L(A)}f^2d\mu_{i_0,i_0}[\delta,\delta]=\mathbf{Re}(\phi(f^2)\delta\mid \delta)=\mathbf{Re}(\phi(f)\delta\mid \phi(f)\delta)\geq 0\ (\neq+\infty), $$
implies that $\mu_{i_0,i_0}[\delta,\delta]$ is  non-negative and finite.

Substituting the equalities above to \eqref{Eq-equivalence-PNM-exp-tA-2} yields \begin{equation}\label{Eq-equivalence-PNM-exp-tA-3}
 \begin{split}
   &(e^{tA}\rvert_L \delta \mid \delta )\\
   =&\int_{\sigma_L(A)}\cos t\lvert x\rvert  d\mu_{i_0,i_0}[\delta,\delta](x) - \int_{\sigma_L(A)}\frac{x}{\lvert x\rvert }\sin t\lvert x\rvert d\mu_{i_0,i_0}[\delta,\delta](x).
 \end{split}
\end{equation}

Set $\gamma$ to be the zero extension of $\mu_{i_0,i_0}[\delta,\delta]$ to the pure imaginary space $\mathbb{H}_I=\mathbb{R}i_1+\mathbb{R}i_2+\mathbb{R}i_3$. Let $\Gamma$ be a non-negative finite regular Borel measure  on $\mathbb{H}_I$ defined by
$$\Gamma(\omega):=\gamma(-\omega), \quad \forall\ \omega\in\mathscr{B}(\mathbb{H}_I). $$
Then \eqref{Eq-equivalence-PNM-exp-tA-3} can be rewritten as
\begin{equation*}
   (U_t\delta\mid \delta)=(e^{tA}\rvert_L \delta \mid \delta )= \int_{\mathbb{H}_I}\cos t\lvert x\rvert d\Gamma(x)+ \int_{\mathbb{H}_I}\frac{x}{\lvert x\rvert }\sin t\lvert x\rvert d\Gamma(x).
\end{equation*}
By the definition of $U_t$, we obtain
\begin{equation}\label{Eq-equivalence-PNM-exp-tA-4}
   \varphi(t)=(U_t\delta \mid \delta )= \int_{\mathbb{H}_I}\cos t\lvert x\rvert d\Gamma(x)+\int_{\mathbb{H}_I}\frac{x}{\lvert x\rvert }\sin t\lvert x\rvert d\Gamma(x).
\end{equation}

{\bf Step 4: }

Consider the slice-condensed measure $\mu'$ of type I given by $$\mu':=\rho_*(\Gamma+\frac{x}{\lvert x\rvert }\Gamma). $$

A direct calculation yields the right side of \eqref{Eq-equivalence-PNM-exp-tA-4} is equal to  $$\int_{\mathbb{R}^+}\cos tx \ d\mathbf{Re}\mu'(x)+\int_{\mathbb{R}^+}\sin tx\  d(\mu'-\mathbf{Re}\mu')(x)$$
Then according to the expression of $\varphi(t)$ in \eqref{Eq-Def-varphi}, we obtain
 $$\int_{\mathbb{R}^+}\cos tx \ d\mathbf{Re}\mu(x)=\int_{\mathbb{R}^+}\cos tx \ d\mathbf{Re}\mu'(x),$$
and
$$\int_{\mathbb{R}^+}\sin tx \ d(\mu-\mathbf{Re}\mu)(x)=\int_{\mathbb{R}^+}\sin tx\  d(\mu'-\mathbf{Re}\mu')(x), $$
for all $t\in\mathbb{R}$.
Hence, by the Stone-Weierstrass theorem, we have
$$\int_{\mathbb{R}^+}f \ d\mathbf{Re}\mu=\int_{\mathbb{R}^+} f\ d\mathbf{Re}\mu'$$
holds for any $f\in C_0(\mathbb{R}^+)$;
and
$$\int_{\mathbb{R}^+}f \ d(\mu-\mathbf{Re}\mu)=\int_{\mathbb{R}^+} f\ d(\mu'-\mathbf{Re}\mu')$$
holds for any $f\in C_0(\mathbb{R}^+)$ with $f(0)=0$.
Thus we know
$$\mathbf{Re}\mu=\mathbf{Re}\mu' \ \text{ on } \mathbb{R}^+,$$
and
$$\mu-\mathbf{Re}\mu=\mu'-\mathbf{Re}\mu' \ \text{ on } \mathbb{R}^+\setminus\{0\}.$$
Since $\mu$ and $\mu'$ are slice-condensed of type II and type I respectively, by definition $\mu-\mathbf{Re}\mu$ and $\mu'-\mathbf{Re}\mu'$ both vanish at the origin, namely,
$$\mu(\{0\})-\mathbf{Re}\mu(\{0\})=\mu'(\{0\})-\mathbf{Re}\mu'(\{0\})=0. $$
Therefore, we come to the final conclusion: $\mu$ is identical with $\mu'$  on $\mathbb{R}^+$.
In other words, any slice-condensed measure of type II is a slice-condensed measure of type I.
\end{proof}

Then the equivalence of Definitions \ref{Def-PNM-1} and \ref{Def-PNM-2}  follows immediately from Lemmas \ref{Lemma-Eq-PNM-1} and \ref{Lemma-Eq-PNM-2}.
\begin{theorem}
  Any slice-condensed measure of type I is a slice-condensed measure of type II, and vice versa.
\end{theorem}

\section{An application to quaternionic random processes}\label{Sec-App}

In terms of applications, a growing popularity of quaternionic random processes has arisen in the field of signal processing (see, e.g. \cite{Buchholz-2008,Navarro-Moreno-2013,Took-2011}). In this section, we shall reveal some mathematical properties of a special family of quaternionic random processes via the  spectral analysis.

Let $\mathbb{E}(Y)$ denote the mean of a quaternionic random variable $Y$. The  covariance  $cov(Y_1,Y_2)$  of arbitrary quaternionic random variables $Y_1,Y_2$ is given as
$$cov(Y_1,Y_2):=\mathbb{E}(Y_1\overline{Y_2});$$
and the variance of $Y$ is defined as
$$var(Y):=cov(Y,Y). $$
One may refer to \cite{Navarro-Moreno-2013,Took-2011} for more basic notations.

\begin{definition}
  A quaternionic  process $X=\{X_t:t\geq 0\}$ is said to be weakly stationary if for all $t,s\geq 0$, and $h>0$, the following equalities hold:
  $$\mathbb{E}(X_t)=\mathbb{E}(X_s),$$
  and
  $$cov(X_t,X_s)=cov(X_{t+h},X_{s+h}). $$
\end{definition}

Thus, $X$ is weakly stationary if and only if it has a constant mean and its auto-covariance $cov(X_t,X_s)$  is a
function of $s-t$ only. For a weakly stationary process $X$, such function is called the auto-covariance function of $X$, and denoted by $c_X$. More precisely,
$$c_X(t):=\left\{\begin{array}{lcl}
                   cov(X_0,X_t), &  & t\geq 0; \\
                   &&\\
                   cov(X_{-t},X_0), & & t<0.
                 \end{array} \right.
$$

Auto-covariance functions have the following property.

\begin{theorem}\label{Thm-PD-WSP}
  Assume that $X$ is a weakly stationary quaternionic processes, then its auto-covariance function $c_X(t)$ is positive definite.
\end{theorem}

\begin{proof}
For all $t_1,t_2,\cdots,t_k\in\mathbb{R}$  and all $p_1,p_2,\cdots,p_k\in\mathbb{H}$,
\begin{equation*}
\begin{split}
   \sum_{1\leq i,j\leq k} \overline{p_i}c_X(t_i-t_j)p_j&= \sum_{1\leq i,j\leq k} \overline{p_i}c_X(t'_j-t'_i)p_j \\
     & =\sum_{1\leq i,j\leq k} \overline{p_i}cov(X_{t'_i},X_{t'_j})p_j \\
     & =var(Y)\\
     & \geq 0,
\end{split}
\end{equation*}
where, $t'_i=-t_i+\displaystyle{\max_{1\leq l\leq k} \{t_l\}}$, and $Y:=\displaystyle{\sum_{i=1}^{k}\overline{p_i}X_{t'_i}}$.
\end{proof}

\begin{theorem}[Spectral theorem for auto-variance functions]\label{Thm-Spec-cov}
Assume $X$ is a weakly stationary quaternionic process, there exists a unique slice-condensed measure $\mu$ on $\mathbb{R}^+$ such that
\begin{equation*}
  c_X(t)=\int_{\mathbb{R}^+}\cos(tx)d\mathbf{Re}\mu(x)+\int_{\mathbb{R}^+}\sin(tx)d(\mu-\mathbf{Re}\mu)(x),
\end{equation*}
whenever the auto-variance function $c_X$ is continuous at the origin.
\end{theorem}

\begin{proof}
 This follows immediately from  Theorems \ref{Thm-PD-WSP} and \ref{Bochner-theorem}.  We need only demonstrate that $c_X$ is  continuous on $\mathbb{R}$.

 Without loss of generality, we assume $$\mathbb{E}(X_t)=0, t\geq 0. $$
 Then a simply application of  the Cauchy-Schwarz inequality yields:
 \begin{equation*}
   \begin{split}
      \lvert c_X(t+\Delta t)-c_X(t)\rvert & =\lvert \mathbb{E}\big(X_0\overline{(X_{t+\Delta t}-X_t)}\big)\rvert \\
        & \leq \sqrt{\mathbb{E}(\lvert X_0\rvert^2)\mathbb{E}(\lvert X_{t+\Delta t}-X_t\rvert^2) }\\
        & =    \sqrt{c(0)[2c(0)-c(\Delta t)-c(-\Delta t)]}
   \end{split}.
 \end{equation*}
 holds for all $t,t+\Delta t\geq 0$, which indicates that $c_X$ is continuous on $\mathbb{R}^+$.
 What's more, it can be easily seen that
 $$c_X(-t)=\mathbb{E}(X_t\overline{X_0})=\overline{\mathbb{E}(X_0\overline{X_t})}=\overline{c_X(t)}$$
 holds for all $t\geq 0$. Hence, $c_X$ is  continuous on $\mathbb{R}$.
\end{proof}

\begin{remark}
  The spectral theorem for auto-variance functions is also valid  for the class of wide-sense stationary quaternion random signals introduced by C. C. Took and D. P. Mandic \cite{Took-2011}, because wide-sense stationarity is stronger than weak stationarity.
\end{remark}

\section{Final remark}\label{Sec-Final}
What we achieve in this paper:

By generalizing Stone's theorem, we manage to reveal some vital spectral characteristics of quaternionic positive definite functions on the real line. We also find an application to weakly stationary quaternionic random processes.

\bigskip

What is the difference between our work and the earlier contributions on the same topic:

Firstly, We would like to give a brief and incomplete presentation of the earlier contribution in the simplest case:

\begin{theorem}\cite{Alpay-2016-1}
Fix $i,j\in\mathbb{S}$ such that $i,j,ij$ form a basis of $\mathbb{H}$. If a quaternion-valued function $\varphi$ on $\mathbb{R}$ is continuous and positive definite, then there exists a unique q-positive measure $\sigma$ on $\mathbb{R}$ such that
$$\varphi(t)=\int_{\mathbb{R}}e^{itx}d\sigma(x)$$
and vice versa.
\end{theorem}
Here, a quaternion-valued measure $\sigma=\sigma_1+\sigma_2j$ with $\sigma_k$ ($k=1,2$) being $\mathbb{C}_i$-valued is said q-positive with respect to $i\in\mathbb{S}$ if $\sigma_k$ ($k=1,2$) satisfy certain symmetric conditions (see Definition 4.1 in \cite{Alpay-2016-1} for more details).

Our main result claims that if a quaternion-valued function $\varphi$ on $\mathbb{R}$ is continuous and positive definite, then there exists a unique slice-condensed measure $\mu$ such that
$$\varphi(t)=\int_{\mathbb{R}^+}\cos(tx)d\mathbf{Re}\mu(x)+\int_{\mathbb{R}^+}\sin(tx)d(\mu-\mathbf{Re}\mu)(x), $$
and vice versa. Compare with the former, it can be easily see that our charaterization is independent with the choice of the imaginary unit $i\in \mathbb{S}$ but less intuitional.

Furthermore, the approach of the earlier contribution is mainly based on a matricial analogue of the Bochner-Minlos theorem as mentioned by D. Alpay and his coauthors, while our approach is based on the functional calculuses, especially the S-functional calculus of unbounded normal operators.

\bigskip

What is the connection between our work and the earlier contribution on the same topic:

The former work and our work together establish a one-to-one correspondence between q-positive measures and slice-condensed measures via quaternionic positive definite functions as intermediary. A direct calculation yields the correspondence: for any q-positive measure $\sigma=\sigma_1+\sigma_2j$ with respect to $i\in\mathbb{S}$, the following equality give a  slice-condensed measures $\mu$:
\begin{equation}\label{Eq-final-remark}
  \mu(x)=\sigma_1(x)+\sigma_1(-x)+i\sigma_1(x)-i\sigma_1(-x)+2i\sigma_2(x)j, \quad
   x\in \mathbb{R}^+.
\end{equation}
such that $\mu$ and $\sigma$ are both corresponding to one quaternionic positive definite function.

\bigskip

What we want to achieve in the future:

In the earlier contributions the quaternionic positive definite functions on locally compact abelian groups are  widely studied. Our work may provide a new perspective on this general setting. Based on what is already known about the simplest case, we can make  reasonable speculations about the general case. For example, one particular speculation is as follows:

Assume $G$ is a locally compact abelian group with Pontryagin dual $\hat{G}$. Then for any continuous quaternionic positive definite function $\varphi$ on $G$, there exists a (not necessarily unique) non-negative finite regular Borel measure $\Gamma$ on $\hat{G}\times\mathbb{S}$ such that
$$\varphi(x)=\int_{\hat{G}\times\mathbb{S}}\mathbf{Re}(\xi(x))+s\mathbf{Im}(\xi(x))d\Gamma(\xi,s). $$
Here $\mathbb S$ denotes the set of quaternionic imaginary units, $\mathbf{Re}$ and $\mathbf{Im}$ mean
extracting the real part and the imaginary part respectively.

This speculation, along with others, will be discussed in our future work.

We acknowledge the support from National Science Foundation for Youth (NSFY) of China (No. 12101094). 

Data availability statement: no data, model, or code were generated or used during the study.

\bigskip

\bibliographystyle{amsplain}

\end{document}